\documentclass{article}

\usepackage{amsmath, amsthm, amsfonts,color}
\usepackage[margin=1in]{geometry}

\newtheorem{thm}{Theorem}[section]
\newtheorem{cor}[thm]{Corollary}
\newtheorem{lem}[thm]{Lemma}
\newtheorem{prop}[thm]{Proposition}
\theoremstyle{definition}

\theoremstyle{remark}

\def\divv{\textrm{div}}
\def\RR{\mathbb{R}}
\def\NN{\mathbb{N}}
\def\Ncal{{\mathcal N}}

\title{Ground state and nodal solutions  for \\ a class of double phase problems}
\author{N.S. Papageorgiou\footnote{National Technical University, Zografou  Campus, 15780 Athens, Greece \& Institute of Mathematics, Physics and Mechanics, 1000 Ljubljana, Slovenia.
E-mail: {\tt npapg@math.ntua.gr}}, 
V.D. R\u adulescu\footnote{AGH University of Science and Technology, 30-059 Krakow, Poland \& Institute of Mathematics, Physics and Mechanics, 1000 Ljubljana, Slovenia
\& University of Craiova, 200585 Craiova, Romania.
E-mail: {\tt vicentiu.radulescu@imfm.si}},  
D.D. Repov\v s\footnote{Faculty of Education and Faculty of Mathematics and Physics, University of Ljubljana \& Institute of Mathematics, Physics and Mechanics, 1000 Ljubljana, Slovenia.
E-mail: {\tt dusan.repovs@guest.arnes.si}}}

\begin{document}
\maketitle

\begin{abstract}
We consider a double phase problem driven by the sum of the $p$-Laplace operator and a weighted $q$-Laplacian ($q<p$), with a weight function which is not bounded away from zero. The reaction term is $(p-1)$-superlinear. Employing the Nehari method, we show that the equation has a ground state solution of constant sign and a nodal (sign-changing) solution.


\end{abstract}

\section{Introduction}
This paper was motivated by several recent contributions to the qualitative analysis of nonlinear problems with unbalanced growth. We mainly refer to the pioneering contributions of Marcellini \cite{marce1,marce2,marce3} who studied lower semicontinuity and regularity properties of minimizers of certain quasiconvex integrals. Problems of this type arise in nonlinear elasticity and are connected with the deformation of an elastic body, cf. Ball \cite{ball1,ball2}.
\subsection{Unbalanced problems and their historical traces}

Let $\Omega$ be a bounded domain in $\RR^N$ ($N\geq 2$) with a smooth boundary. If $u:\Omega\to\RR^N$ is the displacement and $Du$ is the $N\times N$  matrix of the deformation gradient, then the total energy can be represented by an integral of the type
\begin{equation}\label{paolo}I(u)=\int_\Omega f(x,Du(x))dx,\end{equation}
where the energy function $f=f(x,\xi):\Omega\times\RR^{N\times N}\to\RR$ is quasiconvex with respect to $\xi$, see Morrey \cite{morrey}. One of the simplest examples considered by Ball is given by functions $f$ of the type
$$f(\xi)=g(\xi)+h({\rm det}\,\xi),$$
where ${\rm det}\,\xi$ is the determinant of the $N\times N$ matrix $\xi$, and $g$, $h$ are nonnegative convex functions, which satisfy the growth conditions
$$g(\xi)\geq c_1\,|\xi|^p;\quad\lim_{t\to+\infty}h(t)=+\infty,$$
where $c_1$ is a positive constant and $1<p<N$. The condition $p\leq N$ is necessary to study the existence of equilibrium solutions with cavities, that is, minima of the integral \eqref{paolo} that are discontinuous at one point where a cavity forms; in fact, every $u$ with finite energy belongs to the Sobolev space $W^{1,p}(\Omega,\RR^N)$, and thus it is a continuous function if $p>N$. In accordance with these problems arising in nonlinear elasticity, Marcellini \cite{marce1,marce2} considered continuous functions $f=f(x,u)$ with {\it unbalanced growth} that satisfy
$$c_1\,|u|^p\leq |f(x,u)|\leq c_2\,(1+|u|^q)\quad\mbox{for all}\ (x,u)\in\Omega\times\RR,$$
where $c_1$, $c_2$ are positive constants and $1\leq p\leq q$. Regularity and existence of solutions of elliptic equations with $p,q$--growth conditions were studied in \cite{marce2}.

The study of non-autonomous functionals characterized by the fact that the energy density changes its ellipticity and growth properties according to the point has been continued in a series of remarkable papers by Mingione {\it et al.} \cite{1Bar-Col-Min, 2Bar-Col-Min, 8Col-Min, 9Col-Min}. These contributions are in relationship with the work of Zhikov \cite{23Zhikov, 24Zhikov}, which describe the
behavior of phenomena arising in nonlinear
elasticity.
In fact, Zhikov intended to provide models for strongly anisotropic materials in the context of homogenisation.
In particular, he considered the following model
functional
\begin{equation}\label{mingfunc}
{\mathcal P}_{p,q}(u) :=\int_\Omega (|Du|^p+a(x)|Du|^q)dx,\quad 0\leq a(x)\leq L,\ 1<p<q,
\end{equation}
where the modulating coefficient $a(x)$ dictates the geometry of the composite made of
two differential materials, with hardening exponents $p$ and $q$, respectively.

Another significant model example of a functional with $(p,q)$--growth studied by Mingione {\it et al.} is given by
$$u\mapsto \int_\Omega |Du|^p\log (1+|Du|)dx,\quad p\geq 1,$$
which is a logarithmic perturbation of the $p$-Dirichlet energy.

\subsection{Statement of the problem}
Let $\Omega\subseteq\RR^N$ be a bounded domain with a Lipschitz boundary $\partial\Omega$. In this paper we study the following double phase Dirichlet problem
\begin{equation}\label{eq1}
  -\Delta_p u- \divv\left(a(z)|Du|^{q-2}Du\right)=f(z,u) \mbox{ in } \Omega,\;u|_{\partial\Omega}=0,\;1<q<p.
\end{equation}

In this problem, $\Delta_p$ denotes the $p$-Laplace differential operator defined by
$$
\Delta_p u=\divv\left(|Du|^{p-2}Du\right) \mbox{ for all } u\in W^{1,p}_0(\Omega).
$$

So, in problem \eqref{eq1} the differential operator is the sum of a $p$-Laplacian and of a weighted $q$-Laplace operator with $q<p$ and weight $a\in L^{\infty}(\Omega)$, $a(z)>0$ for a.a. $z\in \Omega$. The integrand in the energy functional corresponding to this differential operator is $k(z,t)=\frac{1}{p}t^p+\frac{1}{q} a(z)t^q$ for all $t\geq0$. This integrand exhibits balanced growth since we have for some $c_0>0$
$$
\frac{1}{p}t^p\leq k(z,t)\leq c_0\left(1+t^p\right) \mbox{ for a.a. } z\in \Omega, \mbox{ and all } t>0.
$$

However, the presence of the weight $a(\cdot)$ which is not continuous and is not bounded away from zero, does not permit the use of the nonlinear regularity theory of Lieberman \cite{14Lieberman} and  the nonlinear strong maximum principle of Pucci and Serrin \cite[pp. 111,120]{21Puc-Ser}. In fact, these were the main tools used in the analysis of $(p,q)$-equations (that is, equations driven by the sum of a $p$-Laplacian and of a $q$-Laplacian with no weight) and they led to multiplicity results for such equations. We refer to the works of Mugnai and Papageorgiou \cite{16Mug-Pap}, Papageorgiou and R\u adulescu \cite{17Pap-Rad}, and Papageorgiou, Vetro and Vetro \cite{19Pap-Vet-Vet}, which deal with $(p,q)$-equations with a $(p-1)$-superlinear reaction term. In problem \eqref{eq1} the reaction (source) term $f(z,x)$ is a measurable function which is $C^1$ in the $x$-variable and $(p-1)$-superlinear as $x\rightarrow\pm\infty$. The approach developed in the present paper is based on the Nehari method.

We mention that double phase equations arise in mathematical models of various physical phenomena. We mention the works of Benci, D'Avenia, Fortunato and Pisani \cite{3Ben-D'Av-For-Pis} (in quantum physics), Cherfils and Ilyasov \cite{4Che-Ily} (in reaction-diffusion systems), Zhikov \cite{23Zhikov,24Zhikov} (in nonlinear elasticity theory). Recently, in a series of remarkable papers Mingione {\it et al.} (see \cite{1Bar-Col-Min, 2Bar-Col-Min, 8Col-Min, 9Col-Min}) produced local regularity results for equations driven by the unbalanced double phase operator, namely when
$$
-\divv\left(a(z)|Du|^{p-2}Du\right)-\Delta_q u \mbox{ with } 1<q<p.
$$
In this case the integrand in the corresponding energy functional is
$$
\hat{k}(z,x)=\frac{1}{p}a(z)t^p+\frac{1}{q}t^q \mbox{ for all } t>0
$$
and we have
$$
\frac{1}{q} t^q\leq \hat{k}(z,x)\leq \hat{c}_0\left(1+t^p\right) \mbox{ for a.a. }z\in \Omega,\mbox{ all } t>0,\mbox{ some } \hat{c}_0>0\mbox{ (unbalanced growth)}.
$$

For such equations the appropriate functional space framework is provided by Musielak-Orlicz-Sobolev spaces. This leads to more restrictive conditions on the weight $a(\cdot)$, which cannot be discontinuous and also there are restrictions on the exponents $1<q<p$, namely $\frac{p}{q}<1+\frac{1}{N}$. So, $p$ and $q$ cannot differ too much. We also refer to the works of Colasuonno and Squassina \cite{7Cal-Squ}, Ge, Lv and Lu \cite{11Ge-Lv-Lu}, Liu and Dai \cite{15Liu-Dai}, Papageorgiou, R\u adulescu and Repov\v{s} \cite{prrzamp,prrpams}, Qihu and R\u adulescu \cite{qihu}, and the survey paper of R\u adulescu \cite{22Radulescu}.

\section{Mathematical preliminaries and hypotheses}

The balanced growth of the integrand $k(z,\cdot)$ permits the use of the Sobolev space $W^{1,p}_0(\Omega)$ for the study of problem \eqref{eq1}. By $\|\cdot\|$ we denote the norm of $W^{1,p}_0(\Omega)$. Using the Poincar\'e inequality, we have
$$
\|u\|=\|Du\|_p \mbox{ for all } u\in W^{1,p}_0(\Omega).
$$

Consider the following nonlinear eigenvalue problem
$$
-\Delta_p u=\hat{\lambda}|u|^{p-2}u \mbox{ in } \Omega,\;u|_{\partial\Omega}=0.
$$

It is well-known that this eigenvalue problem has a smallest eigenvalue $\hat{\lambda}_1$, which is isolated, simple and the corresponding eigenfunctions have fixed sign. By $\hat{u}_1$ we denote the positive, $L^p$-normalized (that is, $\|\hat{u}_1\|_p=1$) eigenfunction. Then $\hat{u}_1\in C_0^1(\overline{\Omega})$, $\hat{u}_1(z)>0$ for all $z\in \Omega$ and $\frac{\partial\hat{u_1}}{\partial n}|_{\partial\Omega}<0$, with $n(\cdot)$ being the outward unit normal on $\partial\Omega$. We have
\begin{equation}\label{eq2}
  \hat{\lambda}_1=\inf\left\{\frac{\|Du\|_p^p}{\|u\|_p^p}:\:u\in W^{1,p}_0(\Omega),\:u\not=0\right\},\quad\|D\hat{u}_1\|_p^p=\hat{\lambda}_1\|\hat{u}_1\|_p^p.
\end{equation}
For details we refer to Gasinski and Papageorgiou \cite{10Gas-Pap}.

We introduce the conditions on the weight $a(\cdot)$.

\smallskip\noindent
$H(a)$: $a\in L^\infty(\Omega)$, $a(z)>0$ for a.a. $z\in \Omega$.

\smallskip
We define the following quantity related to the double phase differential operator:
\begin{equation}\label{eq3}
  \vartheta=\inf\left\{\frac{\displaystyle{\|Du\|_p^p+\frac{p}{q}\int_\Omega a(z)|Du|^q dz}}{\|u\|_p^p}:\:u\in W^{1,p}_0(\Omega),\;u\not=0\right\}.
\end{equation}

\begin{lem}\label{lem1}
  If hypothesis $H(a)$ holds, then $\vartheta=\hat{\lambda}_1$.
\end{lem}

\begin{proof}
  By \eqref{eq2}, \eqref{eq3} and hypothesis $H(a)$, it is clear that we have
  \begin{equation}\label{eq4}
    \hat{\lambda}_1\leq \vartheta.
  \end{equation}

  On the other hand, for every $t>0$ we have
  \begin{eqnarray*}
    \vartheta &\leq& \frac{\displaystyle{\|D(t\hat{u}_1)\|_p^p+\frac{p}{q}\int_\Omega a(z)|D(t\hat{u}_1)|^q dz}}{t^p}\ \mbox{ (recall that $\|\hat{u}_1\|_p=1$) } \\
     &=& \|D\hat{u}_1\|_p^p+\frac{1}{t^{p-q}}\frac{p}{q}\int_\Omega a(z)|D\hat{u}_1|^q dz \\
     &=& \hat{\lambda}_1+\frac{1}{t^{p-q}}\frac{p}{q}\int_\Omega a(z)|D\hat{u}_1|^q dz \mbox{ (see \eqref{eq2}).}
  \end{eqnarray*}

  We let $t\rightarrow+\infty$. Since $q<p$, we obtain
  \begin{eqnarray*}
  \vartheta&\leq& \hat{\lambda}_1,\\
  \Rightarrow \vartheta&=& \hat{\lambda}_1 \mbox{ (see \eqref{eq4}). }
  \end{eqnarray*}
  The proof is now complete.
\end{proof}

Next, we introduce the conditions on the reaction function $f(z,x)$. Recall that $p^*$ denotes the critical Sobolev exponent corresponding to $p$, which is defined by
$$
p^*=\left\{
      \begin{array}{ll}
        \frac{Np}{N-p}, & \mbox{ if } p<N \\
        +\infty, & \mbox{ if } N\leq p.
      \end{array}
    \right.
$$
$H(f)$: $f:\Omega\times \RR\rightarrow\RR$ is a measurable function such that for a.a. $z\in \Omega$, $f(z,0)=0$, $f(z,\cdot)\in C^1(\RR\setminus\{0\})$ and
\begin{itemize}
  \item[(i)] $|f'_x(z,x)|\leq a_0(z)\left(1+|x|^{r-2}\right)$ for a.a. $z\in \Omega$ and all $x\in \RR$, with $a_0\in L^\infty(\Omega)$, $p<r<p^*$;
  \item[(ii)] if $F(z,x)=\int_0^x f(z,s)ds$, then $\displaystyle{\lim_{x\rightarrow\pm\infty}\frac{F(z,x)}{|x|^p}}=+\infty$ uniformly for a.a. $z\in \Omega$ and there exist $\tau\in \left(\max\{1,(r-p)\frac{N}{p}\},p^*\right)$ and $\beta_0>0$ such that
    $$
    \beta_0\leq \liminf_{x\rightarrow\pm\infty}\frac{f(z,x)x-pF(z,x)}{|x|^\tau} \mbox{ uniformly fo a.a. }z\in \Omega;
    $$
  \item[(iii)] $\displaystyle{\lim_{x\rightarrow0}\frac{f(z,x)}{|x|^{q-2}x}=0}$ uniformly for a.a. $z\in \Omega$;
  \item[(iv)] $0<(p-1)f(z,x)x\leq f'_x(z,x)x^2$ for a.a. $z\in \Omega$ and all $x\not=0$.
\end{itemize}

 {\bf Remarks}. {\em Hypothesis $H(f)(ii)$ implies that for a.a. $z\in \Omega$, $f(z,\cdot)$ is $(p-1)$-superlinear.
Hypothesis $H(f)(iv)$ implies that for a.a. $z\in \Omega$ we have
$$
x\mapsto \frac{f(z,x)}{|x|^{p-1}} \mbox{ is increasing on } (0,\infty) \mbox{ and on } (-\infty,0),
$$
$$
x\mapsto f(z,x)x-pF(z,x) \mbox{ is increasing in } |x|.
$$

Note that the above monotonicities  are not strict, contrary to what was used in \cite{11Ge-Lv-Lu,15Liu-Dai}.}

\smallskip
 {\bf Examples}. Consider the following functions (for the sake of simplicity, we drop the $z$-dependence):
$$
f_1(x)=|x|^{r-2}x \mbox{ with } p<r<p^*
$$
$$
f_2(x)=\left\{
         \begin{array}{ll}
           |x|^{s-2}x-|x|^{p-2}x, & \mbox{ if } |x|\leq 1 \\
           k|x|^{p-2}x\ln |x|, & \mbox{ if } 1<|x|
         \end{array}
       \right. \mbox{ with } p<s,\;k=s-p>0.
$$

Note that $f_2(\cdot)$ although $(p-1)$-superlinear, does not satisfy the well known Ambrosetti-Rabinowitz condition, which is common in problems with a superlinear reaction (see \cite{15Liu-Dai}).

\smallskip
Let $\varphi:W^{1,p}_0(\Omega)\rightarrow\RR$ be the energy (Euler) functional for problem \eqref{eq1} defined by
$$
\varphi(u)=\frac{1}{p}\|Du\|_p^p+\frac{1}{q}\int_\Omega a(z)|Du|^q dz-\int_\Omega F(z,u) dz \mbox{ for all } u\in W^{1,p}_0(\Omega).
$$

Evidently, $\varphi\in C^1(W^{1,p}_0(\Omega))$. We say that $u\in W^{1,p}_0(\Omega)$ is a ``weak solution" of problem \eqref{eq1}, if
$$
\langle A_p(u),h \rangle+\int_\Omega a(z)|Du|^{q-2}(Du,Dh)_{\RR^N} dz=\int_\Omega f(z,u)h dz \mbox{ for all } h\in W^{1,p}_0(\Omega),
$$
with $\displaystyle{ A_p:W^{1,p}_0(\Omega)\rightarrow W^{-1,p'}(\Omega)=W^{1,p}_0(\Omega)^*}$ $\left(\frac{1}{p}+\frac{1}{p'}=1\right)$ being the nonlinear map defined by
$$
\langle A_p(u),h\rangle=\int_\Omega |Du|^{p-2} (Du,Dh)_{\RR^N} dz \mbox{ for all } u,h\in W^{1,p}_0(\Omega).
$$

As we have already mentioned, our approach is based on the Nehari method. For this reason, we introduce the Nehari manifold for $\varphi$, defined by
$$
{\mathcal N}=\{u\in W^{1,p}_0(\Omega): \langle \varphi'(u),u\rangle=0,\;u\not=0\}.
$$

As above, we denote by $\langle\cdot,\cdot\rangle$  the duality brackets for the pair $\left(W^{-1,p'}(\Omega), W^{1,p}_0(\Omega)\right)$. Evidently, the nontrivial weak solutions of \eqref{eq1} belong to ${\mathcal N}$. Also, since we want to produce nodal solutions, we will use the following set
$$
{\mathcal N}_0=\{u\in W^{1,p}_0(\Omega):\: u^+\in {\mathcal N},\;-u^-\in {\mathcal N}\}.
$$

Recall that $u^+=\max\{u,0\}$, $u^-=\max\{-u,0\}$ for all $u\in W^{1,p}_0(\Omega)$. We have $u^+,u^- \in W^{1,p}_0(\Omega)$ and $u=u^+-u^-$, $|u|=u^+ + u^-$. In the sequel, we will say ``solution" and mean ``weak solution". Also, by $|\cdot|_N$ we denote the Lebesgue measure on $\RR^N$.

\section{Ground state solutions}

In this section we prove the existence of a solution of \eqref{eq1} which minimizes $\varphi|_{N}$. Such a solution is known as a ``ground state solution".

\begin{prop}\label{prop2}
  If hypotheses $H(a)$, $H(f)$ hold and $u\in W^{1,p}_0(\Omega)$, $u\not=0$, then there exists a unique $t_u>0$ such that $t_u u\in {\mathcal N}$.
\end{prop}

\begin{proof}
  We consider the fibering map $\mu_{u}:(0,\infty)\rightarrow\RR$ defined by
$$
\mu_u(t)=\varphi(tu) \mbox{ for all } t>0.
$$

Evidently, $\mu_u\in C^1(0,\infty)$ and using the chain rule, we have
$$
\mu'_u(t)=t^{p-1}\|Du\|_p^p+t^{q-1}\int_\Omega a(z) |Du|^q dz-\int_\Omega f(z,tu)u dz.
$$

We see that
$$
\mu'_u(t)=0\Leftrightarrow tu\in {\mathcal N}.
$$

So, we consider the equation $\mu'_u(t)=0$. This is equivalent to
\begin{equation}\label{eq5}
  \|Du\|_p^p=\int_\Omega \frac{f(z,tu)}{t^{p-1}}u dz-\frac{1}{t^{p-q}}\int_\Omega a(z)|Du|^q dz.
\end{equation}

In relation \eqref{eq5}, the right-hand side is strictly increasing (see hypothesis $H(f)(i)$ and recall that $q<p$). So, there exists a unique $t_u>0$ such that
\begin{eqnarray*}
   && \mu'_u(t_u)=0, \\
   &\Rightarrow& \langle \varphi'(t_u u),u\rangle=0 \mbox{ (by the chain rule), } \\
   &\Rightarrow& \langle\varphi'(t_u u), t_u u\rangle=0, \\
   &\Rightarrow& t_u u\in \Ncal.
\end{eqnarray*}
The proof is complete.
\end{proof}

\begin{cor}
  If hypotheses $H(a)$, $H(f)$ hold, then $\Ncal\not=\emptyset$.
\end{cor}

\begin{prop}\label{prop4}
  If hypotheses $H(a)$, $H(f)$ hold and $u\in \Ncal$, then $\varphi(tu)\leq\varphi(u)$ for all $t>0$.
\end{prop}

\begin{proof}
  We consider the fibering map $\mu_u(\cdot)$ introduced in the proof of Proposition \ref{prop2}. Since $u\in \Ncal$, we have $\mu'_u(1)=0$ and this is the unique critical point of $\mu_u(\cdot)$.

On account of hypotheses $H(f)(i),(ii)$, given $\eta>0$, we can find $c_\eta>0$ such that
\begin{equation}\label{eq6}
  F(z,x)\geq\frac{\eta}{p}|x|^p-c_\eta \mbox{ for a.a. } z\in \Omega \mbox{ and all } x\in \RR.
\end{equation}

Then for any $t>0$ we have
\begin{eqnarray*}
  \varphi(tu) &\leq& \frac{t^p}{p}\|Du\|_p^p +\frac{t^q}{q} \int_\Omega a(z)|Du|^q dz-\frac{t^p}{p}\eta\|u\|_p^p +c_\eta|\Omega|_N \mbox{ (see \eqref{eq6}) }\\
   &=& \frac{t^p}{p}\left[\|Du\|_p^p-\eta\|u\|_p^p\right]+\frac{t^q}{q}\int_\Omega a(z) |Du|^q dz+ c_\eta|\Omega|_N.
\end{eqnarray*}

Choosing $\eta>0$ such that $\|Du\|_p^p<\eta\|u\|_p^p$, we have
\begin{equation}\label{eq7}
  \varphi(tu)\leq c_1 t^q- c_2 t^p + c_\eta|\Omega|_N \mbox{ for some } c_1>0,\;c_2>0.
\end{equation}

Since $q<p$, from \eqref{eq7} we see that
\begin{equation}\label{eq8}
  \mu_u(t)=\varphi(tu)<0 \mbox{ for } t>1 \mbox{ big enough. }
\end{equation}

On the other hand, hypotheses $H(f)(i),(iii)$ imply that given $\varepsilon>0$, we can find $c_\varepsilon>0$ such that
\begin{equation}\label{eq9}
  F(z,x)\leq \frac{\varepsilon}{q}|x|^q+c_\varepsilon|x|^r \mbox{ for a.a. }z\in \Omega \mbox{ and all } x\in \RR.
\end{equation}

Therefore for $t>0$ we have
\begin{eqnarray*}
  \varphi(tu) &\geq& \frac{t^p}{p} \|Du\|_p^p+\frac{t^q}{q} \int_\Omega a(z) |Du|^q dz- \frac{\varepsilon t^q}{q}\|u\|_q^q-c_\varepsilon t^r\|u\|_r^r \mbox{ (see \eqref{eq9}) } \\
  &\geq& \frac{t^p}{p}\|u\|^p+\frac{t^q}{q}\left[\int_\Omega a(z)|Du|^q dz-\varepsilon\|u\|_q^q\right]-c_3 t^r\|u\|^r \mbox{ for some $c_3>0$.}
\end{eqnarray*}

Since $\displaystyle{\int_\Omega a(z)|Du|^q dz>0}$ (see hypothesis $H(a)$ and recall that $u\not=0$), by choosing $\varepsilon>0$ small enough, we have $\displaystyle{\int_\Omega a(z)|Du|^q dz\geq\varepsilon \|u\|_q^q}$ and so
$$
\varphi(tu)\geq c_4 t^p-c_5 t^r \mbox{ for some } c_4,c_5>0.
$$

Since $p<r$, it follows that
\begin{equation}\label{eq10}
  \mu_u(t)=\varphi(tu)>0 \mbox{ for all } t\in(0,1) \mbox{ small enough. }
\end{equation}

From \eqref{eq8} and \eqref{eq10} we can infer that $t_u=1$ is a maximizer of $\mu_u(\cdot)$ and so we have $\varphi(tu)\leq\varphi(u)$ for all $t>0$.
\end{proof}

Let $m=\displaystyle{\inf_{\Ncal} \varphi}$.

\begin{prop}\label{prop5}
  If hypotheses $H(a)$, $H(f)$ hold, then $m>0$.
\end{prop}

\begin{proof}
  Let $u\in W^{1,p}_0(\Omega)\setminus\{0\}$. We have
$$
\varphi(u)\geq \frac{1}{p}\|Du\|_p^p +\frac{1}{q}\left[\int_\Omega a(z)|Du|^q dz-\varepsilon\|u\|_q^q\right]-c_6\|u\|^r \mbox{ for some $c_6>0$ (see \eqref{eq9}).  }
$$

Let $c_7\in \left(0,\frac{1}{p}\right)$. Since $q<p$, for $u\in W^{1,p}_0(\Omega)$ with $0<\|u\|_p\leq 1$, we have
\begin{eqnarray*}
  \varphi(u)\!\! &\geq&\!\!\! \left(\frac{1}{p}-c_7\right)\|Du\|_p^p+ c_7\left[\|Du\|_p^p+\frac{1}{q c_7}\int_\Omega a(z)|Du|^q dz-\varepsilon c_8\|u\|_p^p\right]-c_6\|u\|^r \mbox{ for some $c_8>0$ } \\
   \!\! &=&\!\!\! \left(\frac{1}{p}-c_7\right)\|Du\|_p^p +c_7\left[\|Du\|_p^p+\frac{p}{q}\int_\Omega \frac{a(z)}{c_9}|Du|^q dz-\varepsilon c_8\|u\|_p^p\right]-c_6\|u\|^r \mbox{ with } c_9=\frac{1}{p c_7}>0 \\
   \!\! &\geq&\!\!\! \left(\frac{1}{p}-c_7\right) \|Du\|_p^p +c_7(\hat{\lambda}-\varepsilon c_8)\|u\|_p^p-c_6\|u\|^r \mbox{ (see Lemma \ref{lem1}). }
\end{eqnarray*}

Choosing $\varepsilon\in (0, {\hat{\lambda}_1}/{c_8})$, we obtain
\begin{eqnarray*}
  \varphi(u) &\geq& c_{10}\|u\|^p-c_6\|u\|^r \\
   && \mbox{ for some } c_{10}>0 \mbox{ and all } u\in W^{1,p}_0(\Omega) \mbox{ with } \|u\|_p\leq 1.
\end{eqnarray*}

Because $p<r$, we can find $\rho\in (0,1)$ small, $\rho\leq \frac{1}{\hat{\lambda}_1}$ such that
$$
\varphi(u)\geq\eta_0>0 \mbox{ for all } \|u\|=\rho \mbox{ (note that $\|u\|_p\leq \frac{1}{\hat{\lambda}_1}\|u\|\leq \rho$).}
$$

Now we consider $u\in \Ncal$ and choose $\tau_u>0$ such that $\tau_u\|u\|=\rho$. Then by Proposition \ref{prop4}, we have
\begin{eqnarray*}
   && \varphi(u)\geq \varphi(\tau_u u)\geq \eta_0>0, \\
  &\Rightarrow& m=\inf_{\Ncal} \varphi>0.
\end{eqnarray*}
The proof is now complete.
\end{proof}

The Nehari manifold is much smaller than $W^{1,p}_0(\Omega) $ and so some properties of $\varphi$ which evidently fail globally, can be true for $\varphi|_{\Ncal}$. This is illustrated in the next proposition.

\begin{prop}\label{prop6}
  If hypotheses $H(a)$, $H(f)$ hold, then $\varphi|_{\Ncal}$ is coercive.
\end{prop}

\begin{proof}
  Evidently, it suffices to show that if $\{u_n\}_{n\geq1}\subseteq \Ncal$ and $\varphi(u_n)\leq M$ for some $M>0$ and all $n\in \NN$, then $\{u_n\}_{n\geq 1}\subseteq W^{1,p}_0(\Omega)$ is bounded.

We have
\begin{equation}\label{eq11}
  \|Du_n\|_p^p +\frac{p}{q} \int_\Omega a(z)|Du_n|^q dz-\int_\Omega p F(z,u_n) dz\leq p M \mbox{ for all } n\in \NN.
\end{equation}

Since $u_n\in \Ncal$, we have
\begin{eqnarray}\nonumber
   && \langle \varphi'(u_n), u_n \rangle=0 \mbox{ for all }n\in \NN, \\
  &\Rightarrow& -\|Du_n\|_p^p-\int_\Omega a(z)|Du_n|^q dz+\int_\Omega f(z,u_n) u_n dz=0 \mbox{ for all } n\in \NN. \label{eq12}
\end{eqnarray}

We add \eqref{eq11} and \eqref{eq12} and have
\begin{eqnarray}\nonumber
   && \left(\frac{p}{q}-1\right)\int_\Omega a(z)|Du_n|^q dz+\int_\Omega\left[f(z,u_n)u_n-pF(z,u_n)\right]dz\leq pM \mbox{ for all }n\in \NN,\\
   &\Rightarrow& \int_\Omega \left[f(z,u_n)u_n-pF(z,u_n)\right]dz\leq pM \mbox{ for all } n\in \NN \label{eq13}\\ \nonumber
   &&\mbox{ (recall that $q<p$ and see hypothesis $H(a)$).}
\end{eqnarray}

Hypotheses $H(f)(i),(ii)$ imply that given $\beta_1\in (0,\beta_0)$, we can find $c_{11}=c_{11}(\beta_1)>0$ such that
\begin{equation}\label{eq14}
  \beta_1|x|^\tau -c_{11}\leq f(z,x)x-pF(z,x) \mbox{ for a.a. } z\in \Omega \mbox{ and all } x\in \RR.
\end{equation}

Using \eqref{eq14} in \eqref{eq13}, we obtain that
\begin{equation}\label{eq15}
  \{u_n\}_{n\geq1}\subseteq L^\tau(\Omega) \mbox{ is bounded. }
\end{equation}

We first  assume that $N\not=p$. It is clear by hypothesis $H(f)(ii)$ that without any loss of generality, we may assume that $\tau<r<p^*$. Let $t\in(0,1)$ such that
\begin{equation}\label{eq16}
  \frac{1}{r}=\frac{1-t}{\tau}+\frac{t}{p^*}.
\end{equation}

The interpolation inequality (see for example Papageorgiou and Winkert \cite[p.116]{20Pap-Win}) implies that
\begin{eqnarray}\nonumber
    \|u_n\|_r&\leq& \|u_n\|_\tau^{1-t}\|u_n\|_{p^*}^t \mbox{ for all } n\in \NN, \label{eq17} \\
   \Rightarrow \|u_n\|_r^r&\leq&c_{12}\|u_n\|^{tr} \mbox{ for some } c_{12}>0\mbox{ and all }n\in \NN \\ \nonumber
   && \mbox{ (see \eqref{eq15} and use the Sobolev embedding theorem). }
\end{eqnarray}

Hypothesis $H(f)(i)$ implies that
\begin{equation}\label{eq18}
  f(z,x)x\leq c_{13}\left(1+|x|^r\right) \mbox{ for a.a. } z\in \Omega, \mbox{ all } x\in \RR \mbox{ and some } c_{13}>0.
\end{equation}

From \eqref{eq12} we have
\begin{eqnarray}\nonumber
  \|Du_n\|_p^p +\int_\Omega a(z)|Du_n|^q dz &=& \int_\Omega f(z,u_n)u_n dz  \\ \nonumber
   &\leq& c_{14} \left(1+\|u_n\|_r^r\right) \mbox{ for some } c_{14}>0 \mbox{ and all } n\in \NN \mbox{ (see \eqref{eq18}) } \\ \nonumber
   &\leq& c_{15} \left(1+\|u_n\|^{tr}\right) \mbox{ for some } c_{15}>0 \mbox{ and all } n\in \NN \mbox{ (see \eqref{eq17}), }\\
    \Rightarrow \|u_n\|^p\leq c_{15}\left(1+\|u_n\|^{tr}\right)&&\!\!\!\!\!\!\!\!\!\!\!\! \mbox{ for all } n\in \NN \mbox{ (see hypothesis $H(a)$). } \label{eq19}
\end{eqnarray}

From \eqref{eq16} and the condition on $\tau>1$ (see hypothesis $H(f)(ii)$), we see that $tr<p$. So, from \eqref{eq19} we can infer that
\begin{equation}\label{eq20}
  \{u_n\}_{n\geq1}\subseteq W^{1,p}_0(\Omega) \mbox{ is bounded. }
\end{equation}

Next, assume that $N=p$. In this case $p^*=+\infty$, but by the Sobolev embedding theorem we have
$$
W^{1,p}_0(\Omega)\hookrightarrow L^s(\Omega) \mbox{ for all } 1\leq s<\infty.
$$

So, for the previous argument to work, we need to replace $p^*(=+\infty)$, with $s>r$ so big that
$$
tr=\frac{s(r-\tau)}{s-\tau}<p \mbox{ (see \eqref{eq16} and note that $N=p\Rightarrow r-p<\tau$). }
$$

With such a choice of $s>r$, the previous argument works and we again reach \eqref{eq20}. Therefore we can conclude that $\varphi|_{\Ncal}$ is coercive.
\end{proof}

\begin{prop}\label{prop7}
  If hypotheses $H(a)$, $H(f)$ hold, then we can find $\hat{u}\in \Ncal$ such that $\varphi(\hat{u})=m=\displaystyle{\inf_{\Ncal}\varphi>0}$.
\end{prop}

\begin{proof}
  Let $\{u_n\}_{n\geq1}\subseteq \Ncal$ be a minimizing sequence, that is,
$$
\varphi(u_n)\downarrow m\mbox{ as } n\rightarrow\infty.
$$

From Proposition \ref{prop6}, we have that
$$
\{u_n\}_{n\geq1}\subseteq W^{1,p}_0(\Omega) \mbox{ is bounded. }
$$

So, we may assume that
\begin{equation}\label{eq21}
  u_n\overset{w}{\longrightarrow} \hat{u} \mbox{ in } W^{1,p}_0(\Omega) \mbox{ and } u_n\rightarrow \hat{u} \mbox{ in } L^r(\Omega).
\end{equation}

Since $u_n\in \Ncal$, we have
\begin{equation}\label{eq22}
  \|Du_n\|_p^p+\int_\Omega a(z)|Du_n|^q dz=\int_\Omega f(z,u_n) u_n dz \mbox{ for all } n\in \NN.
\end{equation}

In \eqref{eq22} we pass to the limit as $n\rightarrow\infty$ and use \eqref{eq21} and the weak lower semicontinuity of the norm functional. We obtain
\begin{equation}\label{eq23}
  \|D\hat{u}\|_p^p +\int_\Omega a(z)|D\hat{u}|^q dz\leq \int_\Omega f(z,\hat{u})\hat{u} dz.
\end{equation}

If $\hat{u}=0$, then from \eqref{eq21} and \eqref{eq22}, we see that
\begin{eqnarray*}
   && u_n\rightarrow 0 \mbox{ in } W^{1,p}_0(\Omega), \\
  &\Rightarrow& \varphi(u_n) \rightarrow m=\varphi(0)=0,
\end{eqnarray*}
a contradiction to Proposition \ref{prop5}. Therefore $\hat{u}\not=0$.

If in \eqref{eq23} we have equality, then $\hat{u}\in \Ncal$ and $\varphi(\hat{u})=m$.

So, suppose that
\begin{equation}\label{eq24}
  \|D\hat{u}\|_p^p+\int_\Omega a(z) |D\hat{u}|^q dz<\int_\Omega f(z,\hat{u})\hat{u} dz.
\end{equation}

Using the fibering map $\mu_{\hat{u}}(\cdot)$ from the proof of Proposition \ref{prop4} and from \eqref{eq24} we infer that
\begin{eqnarray}\nonumber
  && \mu_{\hat{u}}(1)<0, \\
  &\Rightarrow& \hat{t}=t_{\hat{u}}\in (0,1) \mbox{ (see \eqref{eq10} and Proposition \ref{prop2}). } \label{eq25}
\end{eqnarray}

We have
\begin{eqnarray*}
  m &\leq& \varphi(\hat{t}\hat{u}) \\
   &=& \frac{1}{p}\|D(\hat{t}\hat{u})\|_p^p+\frac{1}{q}\int_\Omega a(z) |D(t\hat{u})|^q dz-\int_\Omega F(z,\hat{t}\hat{u})dz \\
   &=& \frac{1}{p}\left[\int_\Omega f(z,\hat{t}\hat{u})(\hat{t}\hat{u})dz-\int_\Omega a(z)|D(t\hat{u})|^q dz\right] \\
   &+& \frac{1}{q} \int_\Omega a(z)|D(\hat{t}\hat{u})|^q dz-\int_\Omega F(z,\hat{t}\hat{u})dz \mbox{ (since $\hat{t}\hat{u}\in \Ncal$) }\\
   &=& \int_\Omega \left[\frac{1}{p}f(z,t\hat{u})(t\hat{u})-F(z,\hat{t}\hat{u})\right]dz+\left(\frac{1}{q}-\frac{1}{p}\right)\int_\Omega a(z)|D\hat{u}|^q dz \\
   &<& \int_\Omega \left[\frac{1}{p}f(z,\hat{u})\hat{u}-F(z,\hat{u})\right]dz+\left(\frac{1}{q}-\frac{1}{p}\right)\int_\Omega a(z)|D\hat{u}|^q dz \\
   &&\mbox{ (see \eqref{eq25}, hypothesis $H(f)(iv)$ and recall that $q<p$) } \\
   &\leq& \liminf_{n\rightarrow\infty}\left[\int_\Omega \left[\frac{1}{p}f(z,u_n)u_n-p F(z,u_n)\right]dz+\left(\frac{1}{q}-\frac{1}{p}\right)\int_\Omega a(z)|Du_n|^q dz\right]\\
   && \mbox{ (see \eqref{eq21} and use Lemma 1.12 of Heinonen, Kilpel\"ainen and Martio \cite[p.16]{12Hei-Kil-Mar}) }\\
   &=&m,
\end{eqnarray*}
a contradiction. Therefore we conclude that $\hat{u}\in \Ncal$ and $\varphi(\hat{u})=m$.
\end{proof}

The next proposition shows that the Nehari manifold is a natural constraint for $\hat{u}$ (see Papageorgiou, R\u adulescu and Repov\v s \cite[p.425]{18Pap-Rad-Rep}).  In what follows, we denote
$$
K_\varphi=\left\{u\in W^{1,p}_0(\Omega):\:\varphi'(u)=0\right\} \mbox{ (the critical set of $\varphi$). }
$$

\begin{prop}\label{prop8}
  If hypotheses $H(a)$, $H(f)$ hold and $\hat{u}\in \Ncal$ satisfies $\varphi(\hat{u})=m$, then $\hat{u}\in K_\varphi$, $\hat{u}$ is a solution of \eqref{eq1}, $\hat{u}\in W^{1,p}_0(\Omega)\cap L^{\infty}(\Omega)$ and $\hat{u}$ does not change sign.
\end{prop}

\begin{proof}
  Consider the function $k:W^{1,p}_0(\Omega)\rightarrow\RR$ defined by
$$
k(u)=\|Du\|_p^p+\int_\Omega a(z)|Du|^q dz- \int_\Omega f(z,u)u dz \mbox{ for all } u\in W^{1,p}_0(\Omega).
$$

Evidently, $k\in C^1(W^{1,p}_0(\Omega))$ and we have
$$
\langle k'(u),h\rangle=p\langle A_p(u),h\rangle+q\int_\Omega a(z)|Du|^{q-2}(Du,Dh)_{\RR^N}dz-\int_\Omega \left[f'_x(z,u)u+f(z,u)\right]h dz
$$
for all $h\in W^{1,p}_0(\Omega)$.

From Proposition \ref{prop7}, we know that
$$
\varphi(\hat{u})=m=\inf\left\{\varphi(u):\:k(u)=0,\;u\in W^{1,p}_0(\Omega)\setminus\{0\}\right\}.
$$

Then by the Lagrange multiplier rule (see Papageorgiou, R\u adulescu and Repov\v s \cite[Theorem 5.5.9]{18Pap-Rad-Rep}), we can find $\vartheta\geq0$ such that
\begin{eqnarray}
  && \varphi'(\hat{u})+\vartheta k'(\hat{u})=0 \mbox{ in } W^{-1,p'}(\Omega)=W^{1,p}_0(\Omega)^*, \label{eq26} \\ \nonumber
  &\Rightarrow& \langle \varphi'(\hat{u}),\hat{u}\rangle+\vartheta\langle k'(\hat{u}),\hat{u}\rangle=0, \\ \nonumber
  &\Rightarrow& \vartheta\langle k'(\hat{u}),\hat{u}\rangle=0 \mbox{ (since $\hat{u}\in \Ncal$). }
\end{eqnarray}

If $\vartheta\not=0$, then we must have
$$
\langle k'(\hat{u}),\hat{u}\rangle=0,
$$
\begin{eqnarray*}
  &\Rightarrow& p\|D\hat{u}\|_p^p+q\int_\Omega a(z)|D\hat{u}|^q dz-\int_\Omega f(z,\hat{u})\hat{u} dz=\int_\Omega f'_x(z,\hat{u})\hat{u} dz, \\
  &\Rightarrow& p\left[\|D\hat{u}\|_p^p+\int_\Omega a(z)|Du|^q dz-\int_\Omega f(z,\hat{u})\hat{u}dz\right]+(q-p)\int_\Omega a(z)|Du|^q dz \\
   &=& \int_\Omega \left[f'_x(z,\hat{u})\hat{u}^2-(p-1)f(z,\hat{u})\hat{u}\right]dz, \\
   &\Rightarrow& 0>(q-p)\int_\Omega a(z)|Du|^p dz\geq \int_\Omega \left[f'_x(z,\hat{u})\hat{u}^2-(p-1)f(z,\hat{u})\hat{u}\right]dz\geq0 \\
  && \mbox{ (recall that $\hat{u}\in \Ncal$, $q<p$ and see hypotheses $H(a)$, $H(f)(iv)$), }
\end{eqnarray*}
a contradiction. Therefore $\vartheta=0$ and so from \eqref{eq26} we have
\begin{eqnarray}
   && \varphi'(\hat{u})=0 \mbox{ in } W^{-1,p'}(\Omega), \label{eq27}\\ \nonumber
   &\Rightarrow& \hat{u}\in K_\varphi \mbox{ and } \hat{u} \mbox{ is a solution of \eqref{eq1}.}
\end{eqnarray}

Invoking Theorem 7.1 of Ladyzhenskaya and Uraltseva \cite[p.286]{13Lad-Ura} we have
$$
\hat{u}\in W^{1,p}_0(\Omega)\cap L^\infty(\Omega).
$$

We claim that $\hat{u}\in \Ncal$ has fixed sign. Arguing indirectly, suppose that $\hat{u}$ is nodal (sign-changing). Then $u^{\pm}\not\equiv0$. From \eqref{eq27}, we have
\begin{equation}\label{eq28}
  \langle \varphi'(\hat{u}),h\rangle=0 \mbox{ for all } h\in W^{1,p}_0(\Omega).
\end{equation}

In \eqref{eq28} we first choose $h=\hat{u}^+\in W^{1,p}_0(\Omega)$ and then $h=-\hat{u}^{-}\in W^{1,p}_0(\Omega)$. We obtain
\begin{eqnarray*}
  && \langle \varphi'(\hat{u}^+),\hat{u}^+\rangle=0 \mbox{ and } \langle\varphi'(-\hat{u}^-),-\hat{u}^-\rangle=0, \\
  &\Rightarrow& \hat{u}^+ \mbox{ and } -\hat{u}^-\in \Ncal.
\end{eqnarray*}

We have
\begin{eqnarray*}
   && m=\varphi(\hat{u})=\varphi(\hat{u}^+)+\varphi(-\hat{u}^-)\geq 2m, \\
   &\Rightarrow& m=0, \mbox{ a contradiction since $m>0$ (see Proposition \ref{prop5}). }
\end{eqnarray*}

We conclude that $\hat{u}$ must have fixed sign.
\end{proof}

\section{Nodal solutions}

In this section we produce a nodal solution for problem \eqref{eq1}. To this end, we employ the following set, which contains all nodal solutions of \eqref{eq1}
$$
\Ncal_0=\{y\in W^{1,p}_0(\Omega):\:y^+\in \Ncal,\, -y^-\in \Ncal\}.
$$

\begin{prop}\label{prop9}
  If hypotheses $H(a)$, $H(f)$ hold, then $\Ncal_0\not=\emptyset$.
\end{prop}

\begin{proof}
  Consider $y\in W^{1,p}_0(\Omega)$ such that $y^+\not\equiv0$, $y^-\not\equiv0$ (that is, a nodal Sobolev function). According to Proposition \ref{prop2}, we can find $t_{\pm}>0$ such that
$$
t_+ y^+\in \Ncal \mbox{ and } t_-y^-\in \Ncal .
$$

We set
$$
v=t_+y^+-t_-y^-.
$$

Evidently, $v^+=t_+ y^+\in \Ncal$ and $v^-=t_-y^-\in \Ncal$. Therefore $v\in \Ncal_0\not=\emptyset$.
\end{proof}

We set
$$
m_0=\inf_{\Ncal_0} \varphi.
$$

\begin{prop}\label{prop10}
  If hypotheses $H(a)$, $H(f)$ hold, then there exists $y_0\in \Ncal_0$ such that $m_0=\varphi(y_0)$.
\end{prop}

\begin{proof}
  Let $\{y_n\}_{n\geq1}\subseteq \Ncal_0$ be a minimizing sequence, that is,
\begin{equation}\label{eq29}
  \varphi(y_n)\downarrow m_0 \mbox{ as } n\rightarrow\infty.
\end{equation}

We have
\begin{eqnarray*}
  &&\varphi(y_n) = \varphi(y_n^+)+\varphi(-y_n^-)\geq 2m>0 \mbox{ for all $n\in \NN$ (see Proposition \ref{prop5}), } \\
  &\Rightarrow& m_0\geq2m>0, \\
  &\Rightarrow& \{\varphi(y_n^+)\}_{n\geq1} \mbox{ and } \{\varphi(-y_n^-)\}_{n\geq1} \mbox{ are bounded. }
\end{eqnarray*}

Then on account of Proposition \ref{prop6}, we see that
$$
\left\{y_n^+\right\}_{n\geq1} \subseteq W^{1,p}_0(\Omega) \mbox{ and } \left\{y_n^-\right\}_{n\geq1}\subseteq W^{1,p}_0(\Omega) \mbox{ are bounded. }
$$

So, we may assume that
$$
y_n^+\overset{w}{\longrightarrow} v_1 \mbox{ and } y_n^-\overset{w}{\longrightarrow} v_2 \mbox{ in } W^{1,p}_0(\Omega).
$$

As in the proof of Proposition \ref{prop7} we show that $v_1\not=0$, $v_2\not=0$. Invoking Proposition \ref{prop7}, we can find $t_1>0$, $t_2>0$ such that
\begin{equation}\label{eq30}
  t_1v_1\in \Ncal \mbox{ and } t_2v_2\in \Ncal.
\end{equation}

We set
\begin{equation}\label{eq31}
  y_0=t_1v_1-t_2v_2 \mbox{ with } y_0^+=t_1v_1,\;y_0^-=t_2v_2.
\end{equation}

We have
\begin{eqnarray*}
m_0&=&\lim_{n\rightarrow\infty} \varphi(y_n) \mbox{ (see \eqref{eq29}) }\\
&=&\lim_{n\rightarrow\infty}\left[\varphi(y_n^+)+\varphi(-y_n^-)\right]\\
&\geq&\liminf_{n\rightarrow\infty} \left[\varphi(t_1y_n^+)+\varphi(-t_2y_n^-)\right]\\
&&\mbox{ (see Proposition \ref{prop4} and recall that $y_n^+,-y_n^-\in \Ncal$)}\\
&\geq& \varphi(tv_1)+\varphi(-t_2v_2)\\
&&\mbox{ (from the sequential weak lower semicontinuity of $\varphi$) }\\
&=&\varphi(y_0) \mbox{ (see \eqref{eq31}) }\\
&\geq& m_0 \mbox{ (since $y_0\in \Ncal_0$, see \eqref{eq30}, \eqref{eq31}), }\\
\Rightarrow \varphi(y_0)&=&m_0,\ y_0\in \Ncal_0.
\end{eqnarray*}
The proof is now complete.
\end{proof}

We show that $\Ncal_0$ is a natural constraint for $\varphi$.
To this end, we will use some tools from Nonsmooth Analysis. In particular, we will use the generalized subdifferential in the sense of Clarke \cite{5Clarke}. Let us recall its definition. Suppose that $X$ is a Banach space and $\Psi:X\rightarrow\RR$ is a locally Lipschitz function. For every $u,h\in X$, we define
$$
\Psi^0(u;h)=\underset{t\downarrow0}{\limsup_{x\rightarrow u}}\frac{\Psi(x+th)-\Psi(x)}{t}\,.
$$

Then the mapping $h\mapsto \Psi^0(u;h)$ is continuous and sublinear. We define the set
$$
\partial\Psi(u)=\{u^*\in X^*:\langle u^*,h\rangle_X\leq\Psi^0(u;h)\mbox{ for all }h\in X\}
$$
with $\langle\cdot,\cdot\rangle_X$ being the duality brackets for the pair $(X,X^*)$. By the Hahn-Banach theorem, $\partial\Psi(u)\not=\emptyset$ for all $u\in X$ and is convex and $w^*$-compact. The multifunction $u\mapsto\partial\Psi(u)$ is the generalized subdifferential of $\Psi(\cdot)$. This notion has a very rich calculus, extending the smooth calculus and that of Convex Analysis (see Clarke \cite{5Clarke, 6Clarke}).

\begin{prop}\label{prop11}
  If hypotheses $H(a)$, $H(f)$ hold, then $y_0\in K_\varphi$ and so $y_0\in W^{1,p}_0(\Omega)\cap L^\infty(\Omega)$ is a nodal solution of problem \eqref{eq1}.
\end{prop}

\begin{proof}
  With $k:W^{1,p}_0(\Omega)\rightarrow\RR$ as in the proof of Proposition \ref{prop8}, we have
$$
\varphi(y_0)=m_0=\inf\left\{\varphi(y):k(y^+)=0,\; k(-y^-)=0,\;y^{\pm}\not=0 \right\}
$$
$$
\mbox{ (see Proposition \ref{prop10}).}
$$

In this case we cannot apply the classical multiplier rule since the functions $y\mapsto k_1(y)=k(y^+)$ and $y\mapsto k_2(y)=k(-y^-)$ are no longer of class $C^1$. However, they are locally Lipschitz and so instead we can use the nonsmooth multiplier rule of Clarke \cite[p.221]{6Clarke}. So, we can find $\vartheta_1,\vartheta_2\geq0$ such that
\begin{equation}\label{eq32}
  0\in \partial[\varphi+\vartheta_1 k_1+\vartheta_2 k_2](y_0)
\end{equation}
with $\partial[\varphi+\vartheta_1 k_1+\vartheta_2 k_2]$ being the generalized subdiffrerential in the sense of Clarke of the locally Lipschitz function $u\mapsto \varphi(u)+\vartheta_1 k_1(u)+\vartheta_2 k_2(u)$ defined above. From the sum rule of the subdifferential calculus (see Clarke \cite[p.200]{6Clarke}) and \eqref{eq32} we have
\begin{equation}\label{eq33}
  0=\varphi'(y_0)+\vartheta_1 h_1^*+\vartheta_2 h_2^*
\end{equation}
with $h_1^*\in \partial k_1(y_0)$ and $h_2^*\in \partial k_2(y_0)$. On \eqref{eq33} we act with $y_0$ and obtain
\begin{equation}\label{eq34}
  0=\vartheta_1\langle h_1^*,y_0^+ \rangle+\vartheta_2\langle h_2^*,-y^-_0\rangle
\end{equation}
$$
\mbox{ (note that $\varphi'(y_0)=\varphi'(y_0^+)+\varphi'(-y^-_0)$ and recall that $y^+_0,-y^-_0\in \Ncal$). }
$$

From the definition of the Clarke generalized subdifferential and the subdifferential calculus (see Clarke \cite[pp. 42,76]{5Clarke}), we have
\begin{eqnarray}\nonumber
   && \vartheta_1\langle h^*_1, y_0^+ \rangle \\ \nonumber
   &\leq& \vartheta_1\left[p\|Dy_0^+\|_p^p+q\int_\Omega a(z)|Dy_0^+|^q dz-\int_\Omega\left[f'_x(z,y_0^+)(y_0^+)^2+f(z,y_0^+)y_0^+\right]dz\right] \\ \nonumber
   &=& \vartheta_1\left(p\left[\|Dy_0^+\|_p^p+\int_\Omega a(z)|Dy_0^+|^q dz-\int_\Omega f(z,y_0^+)y_0^+ dz\right]+(q-p)\int_\Omega a(z) |Dy_0^+|dz\right) \\ \nonumber
   &-& \vartheta_1\int_\Omega \left[f'_x(z,y_0^+)(y_0^+)^2-(p-1)f(z,y_0^+)y_0^+\right] dz \\
   &\leq& 0 \mbox{ (since $y_0^+\in \Ncal$, $q<p$ and using hypothesis $H(f)(iv)$). } \label{eq35}
\end{eqnarray}

Similarly we show that
\begin{equation}\label{eq36}
  -\vartheta_2\langle h_2^*,-y^-_0\rangle\geq0.
\end{equation}

From \eqref{eq34}, \eqref{eq35} and \eqref{eq36} we infer that
$$
\vartheta_1\langle h_1^*,y_0^+\rangle=0 \mbox{ and } \vartheta_2\langle h_2^*,y_0^-\rangle=0.
$$

If $\vartheta_1\not=0$, then $\langle h_1^*, y_0^+\rangle=0$ and since $y_0^+\in \Ncal$, as in the proof of Proposition \ref{prop8}, we have a contradiction. Hence $\vartheta_1=0$. Similarly we show that $\vartheta_2=0$. Therefore we finally have
\begin{eqnarray*}
  && \varphi'(y_0)=0 \mbox{ (see \eqref{eq33}), } \\
  &\Rightarrow& y_0\in K_\varphi, \\
  &\Rightarrow& y_0\in W^{1,p}_0(\Omega)\cap L^\infty(\Omega) \mbox{ is a nodal solution \eqref{eq1}. }
\end{eqnarray*}
The proof is now complete.
\end{proof}

So, we can finally state the following multiplicity theorem for problem \eqref{eq1}.

\begin{thm}
  If hypotheses $H(a)$, $H(f)$ hold, then problem \eqref{eq1} has a ground state solution $\hat{u}\in W^{1,p}_0(\Omega)\cap L^\infty(\Omega)$ with fixed sign and a nodal solution $y_0\in W^{1,p}_0(\Omega)\cap L^\infty(\Omega)$.
\end{thm}

\medskip
{\bf Acknowledgments.} This research was supported by the Slovenian Research Agency grants
P1-0292, J1-8131, N1-0064, N1-0083, and N1-0114.

\end{document}